\documentclass{amsart}

\usepackage{amssymb, latexsym, amsmath, amsthm}

\newtheorem{theorem}{Theorem}

\begin{document}

\title{The number of group homomorphisms from $D_m$ into $D_n$}
\author{Jeremiah W. Johnson}
\address{Department of Mathematics and Computer Science, Penn State Harrisburg, Middletown PA 17057}
\email{jwj10@psu.edu}

\begin{abstract} We derive general formul\ae\hspace{0.005in} for counting the number of homomorphisms between dihedral groups using only elementary  group theory.
\end{abstract}

\maketitle

This note considers the problem of counting the number of group homomorphisms from $D_m$ into $D_n$, where for a positive integer $l$, $D_l$ denotes the finite group generated by two generators $r_l$ and $f_l$ subject to the relations $r_l^l = e = f_l^2$ and $r_lf_l = f_lr_l^{-1}$. We derive some general formul\ae\hspace{0.005in} using only elementary  group theory and a few basic facts about the dihedral groups. We will assume throughout that $\phi$ represents Euler's totient function. 

\begin{theorem}\label{THM:ODD} Let $m$ and $n$ be positive odd integers. The number of group homomorphisms from $D_m$ into $D_n$ is 
\begin{equation}
	1+n\left(\sum_{k|\gcd(m,n)}\phi(k)\right).
\end{equation}    
\end{theorem}

\begin{proof}
Suppose that $\rho\colon D_m \to D_n$ is a group homomorphism, where $m$ and $n$ are positive odd integers. We consider all of the places that $\rho$ could send the generators $r_m$ and $f_m$ of $D_m$ which yield group homomorphisms. As $m$ is odd, it must be the case that $\rho(r_m) = r_n^{\alpha}$, where $r_n^{\alpha}$ is an element of $D_n$ whose order divides both $m$ and $n$. Let $k$ represent the order of this element. There are precisely $\phi(k)$ elements of order $k$ in $D_n$. Since $\rho$ can send $r_m$ to any one of these elements, we have $\sum_{k | m, n} \phi(k)$ choices for $\rho(r_m)$.

Next, consider our choices for $\rho(f_m)$. Since $|\rho(f_m)|$ divides $|f_m| = 2$, either $\rho(f_m) = r_n^{\beta}f_n$, $0 \leq \beta < n$, or $\rho(f_m) = e_n$. But not all of these choices for $\rho(f_m)$ yield homomorphisms, as can be seen when we consider where $\rho$ sends the remaining elements in $D_m$ of the form $r_m^kf_m$, where $0 < k < m$. If $\rho(f_m) = e_n$ and $\rho(r_m) = r_n^{\alpha}$, where $\alpha \neq 0$ or $n$, then $\rho(r_mf_m) = r_n^{\alpha}e_n = r_n^{\alpha}$, and $|r_n^{\alpha}|$ does not divide $|r_mf_m|$. Therefore, if $\rho(f_m) = e_n$, then $\rho$ must be trivial. Conversely, when $\rho(f_m) = r_n^{\beta}f_n$, $\rho(r_m^kf_m) = r_n^{k\alpha +\beta \mod n}f_n$, and $|r_n^{k\alpha +\beta \mod n}f_n|$ divides $|r_m^kf_m|$. So, given any choice for $r_m$, we have $n$ choices for $f_m$. Including the trivial homomorphism gives the result.    
\end{proof}

When $m$ and $n$ are positive odd integers and $m|n$, it follows from the fact that $\sum_{k|n}\phi(k) = n$ \cite{WS88} that there are $mn+1$ group homomorphisms from $D_m$ into $D_n$, and furthermore, there are  $n^2+1$ group endomorphisms of $D_n$.

When $m$ is a positive odd integer and $n$ is a positive even integer, $r_n^{n/2}$ is a possible choice for the image of $f_m$. However, if $f_m$ is sent to $r_n^{n/2}$, then the image of $r_m$ must be $e_n$; otherwise the map fails to be a homomorphism. Again let $\rho\colon D_m\to D_n$ denote the map and suppose that $\rho(r_mf_m) = r_n^{\alpha}r_n^{n/2}$ for some $\alpha \neq 0$ or $n$ This element necessarily has order not equal to 2 or 1; a contradiction. So in this case, we gain a single additional map sending $r_m$ to $e_n$ and $f_m$ to $r_n^{n/2}$. Taking this additional consideration into account, a proof nearly identical to that used for Theorem \ref{THM:ODD} yields the following result. 

\begin{theorem}\label{THM:ODDEVEN} Let $m$ be a positive odd integer and $n$ a positive even integer. The number of group homomorphisms from $D_m$ into $D_n$ is 
\begin{equation} 2+n\left(\sum_{k|\gcd(m, n)} \phi(k)\right).
\end{equation}
\end{theorem}

When $m$ is a positive even integer, the number of choices that exist for the image of $r_m$ includes all elements of the form $r_n^kf_n$, $0 \leq k < n$. This creates a number of additional possibilities.  

\begin{theorem}\label{THM:EVEN} Let $m$ and $n$ be positive even integers. The number of group homomorphisms from $D_m$ into $D_n$ is 
\begin{equation}
	4+4n+n\left(\sum_{k|\gcd(m,n)}\phi(k)\right).
\end{equation}
\end{theorem}

\begin{proof} Suppose that $\rho\colon D_m \to D_n$ is a group homomorphism, where $m$ and $n$ are positive even integers. When $m$ is even, we have in addition to the $\sum_{k | m, n} \phi(k)$ possible choices for $\rho(r_m)$ that occur when $m$ is odd the possibility of mapping $r_m$ to those elements in $D_n$ of the form $r_n^{\beta}f_n$. As there are $n$ such elements of the latter type, we have $\sum_{k | m, n} \phi(k) + n$ possible choices for $\rho(r_m)$.

Next, suppose $\rho(r_m) = r_n^{\alpha}$ and consider $\rho(f_m)$. Since $|\rho(f_m)|$ divides $|f_m| = 2$, it must be the case that either $\rho(f_m) = r_n^{\beta}f_n$, $0 \leq \beta < n$, $\rho(f_m) = r_n^{n/2}$, or $\rho(f_m) = e_n$. If $\alpha = 0$ or $n/2$, any of these $n+2$ choices for $\rho(f_m)$ will yield a homomorphism. If $\alpha \neq 0$ or $n/2$, then $\rho(f_m)$ cannot equal $e_n$ or $r_n^{n/2}$. So, there are $n\left(\sum_{k|\gcd(m,n)}\phi(k)\right)+4$ homomorphisms sending $r_m$ to an element of the form $r_n^{\alpha}$.

Assume next that $\rho(r_m) = r_n^{\alpha}f_n$. Since $|\rho(r_m)| = |\rho(f_m)| = 2$, it follows that if $\rho$ is a homomorphism, then the size of the image of $\rho$ is either 2 or 4. There is only one subgroup of each order containing $r_n^{\alpha}f_m$; the cyclic subgroup $\langle r_n^{\alpha}f_m\rangle$, and the subgroup $\langle r_n^{\alpha}f_m, r_n^{\alpha+n/2 \mod n}f_n\rangle$. There are two choices for $f_m$ which result in the first case; namely, $\rho(f_m) = e_n$, or $\rho(f_m) = r_n^{\alpha}$. Similarly, there are two choices for $f_m$ which result in the second case; $\rho(f_m) = r_n^{\alpha+n/2 \mod n}f_n$ or $\rho(f_m) = r^{n/2}$. A brief calculation shows that each of these four possibilities does in fact give a homomorphism, which leads to the conclusion.
\end{proof}

When $m$ and $n$ are positive even integers and $m|n$, it follows that the number of group homomorphisms from $D_m$ into $D_n$ is $4+4n+mn$, while the number of group endomorphisms of $D_n$ is $(n+2)^2$. 

The last case to consider is when $m$ is even and $n$ is odd.

\begin{theorem}\label{THM:EVENODD} Let $m$ be a positive even integer and $n$ a positive odd integer. The number of group homomorphisms from $D_m$ into $D_n$ is 
\begin{equation} 1 + 2n + n\left(\sum_{k|\gcd(m, n)}\phi(k)\right).
\end{equation}
\end{theorem}

\begin{proof} As in the proof of Theorem \ref{THM:ODD}, there are $n\left(\sum_{k|\gcd(m,n)}\phi(k)\right)$ homomorphisms in which $r_m$ is sent to an element of the form $r_n^{\alpha}$, $0 < \alpha < n$, plus the trivial homomorphism. In addition, we could send $r_m$ to any of the $n$ elements of the form $r_n^{\alpha}f_n$, $0 \leq \alpha < n$. If $\rho(r_m) = r_n^{\alpha}f_n$, then the image of $\rho$ is a subgroup of order 2, the cyclic subgroup $\langle \rho(r_m)\rangle$. That leaves two choices for $\rho(f_m)$; either $\rho(f_m) = e_n$ or $\rho(f_m) = r_n^{\alpha}f_n$, from which the result follows.   
\end{proof}

When $\gcd(m,n) = 1$, Theorems \ref{THM:ODDEVEN} and \ref{THM:EVENODD} lead to the succinct formul\ae\hspace{0.005in} that the number of group homomorphisms from $D_m$ into $D_n$ equals $n+2$ when $m$ is odd and $n$ is even, and $3n+1$ when $m$ is even and $n$ is odd.

\end{document}